\documentclass[11pt,reqno]{amsart}
\usepackage{graphicx}
\usepackage{amsmath,amssymb,amsfonts,euscript,enumerate}
\usepackage{amsthm}
\usepackage{color,wrapfig}
\usepackage{pxfonts}
\usepackage{verbatim}
\newtheorem{thm}{Theorem}[section]

\newtheorem{lemma}[thm]{Lemma}
\newtheorem{prop}[thm]{Proposition}

\newcommand{\R}{{\mathbb{R}}}
\newcommand{\Z}{{\mathbb{Z}}}

\newcommand{\bs}{\backslash}

\newcommand{\3}{\varepsilon}
\newcommand{\4}{\widetilde}

\begin{document}

\title{Singular limit of the generalized Burgers equation with absorption}

\author[Kin Ming Hui]{Kin Ming Hui}
\address{Kin Ming Hui:
Institute of Mathematics, Academia Sinica,\\
Taipei, 10617, Taiwan, R.O.C.}
\email{kmhui@gate.sinica.edu.tw}

\author[Sunghoon Kim]{Sunghoon Kim}
\address{Sunghoon Kim:
Department of Mathematics, School of Natural Sciences, The Catholic University of Korea,
43 Jibong-ro, Wonmi-gu, Bucheon-si, Gyeonggi-do, 420-743, Republic of Korea}
\email{math.s.kim@catholic.ac.kr}

\date{Jan 7, 2015}

\begin{abstract}
We prove the convergence of the solutions $u_{m,p}$ of the equation $u_t+(u^m)_x=-u^p$ in $\R\times (0,\infty)$, $u(x,0)=u_0(x)\ge 0$ in $\R$,  as $m\to\infty$ for any $p>1$ and $u_0\in L^1(\R)\cap L^{\infty}(\R)$ or as $p\to\infty$ for any $m>1$ and $u_0\in L^{\infty}(\R)$ . We also show that in general $\underset{p\to\infty}\lim\underset{m\to\infty}\lim u_{m,p}\ne\underset{m\to\infty}\lim\underset{p\to\infty}\lim u_{m,p}$. 

\end{abstract}

\maketitle
\vskip 0.2truein

\setcounter{equation}{0}
\setcounter{section}{0}

\section{Introduction}\label{section-intro}
\setcounter{equation}{0}
\setcounter{thm}{0}

Recently there is a lot of studies on the singular limit of solutions of partial differential equations. 
Singular limit of solutions of the porous medium equation,
\begin{equation}
\left\{\begin{aligned}
u_t=&\Delta u^m\qquad\mbox{ in }\R^n\times (0,T)\\
u(x,0)=&u_0\ge 0\quad\mbox{ in }\R^n
\end{aligned}\right.
\end{equation}
as $m\to\infty$ is proved by L.A.~Caffarelli and A.~Friedman in \cite{CF} when $u_0$ satisfies some appropriate conditions. Later P.~B\'enilan, L.~Boccardo and M.~Herrero \cite{BBH} and P.E.~Sacks \cite{S} extended this result to more general initial value $0\le u_0\in L^1(\R^n)$. Singular limits of the solutions of the porous medium equation with absorption or drift term were proved by K.M.~Hui in \cite{H1}, \cite{H2} and \cite{H3}. Singular limit as $p\to\infty$ of the solutions of the one dimensional nonlinear wave equation
\begin{equation}
\phi_{tt}-\phi_{xx}=-|\phi|^{p-1}\phi
\end{equation}
with initial data $\phi(x,0)=\phi_0(x)$, $\phi_t(x,0)=\phi_1(x)$, was proved by T.~Tao in \cite{T}. Singular limit of solutions of the hyperbolic equation
\begin{equation}\label{burgers-eqn}
\left\{\begin{aligned}
u_t+(u^m)_x=&0\qquad\quad\,\,\mbox{ in }\R\times (0,\infty)\\
u(x,0)=&u_0\ge 0\quad\mbox{ in }\R
\end{aligned}\right.
\end{equation}
as $m\to\infty$ was proved by X.~Xu in \cite{X}. Recently B.~Perthame, F.~Quiros and J.L.~Vazquez \cite{PQV} 
proved the singular limit of solutions of the following system of equations, which arises in the Hele-Shaw models of tumor growth \cite{P}, \cite{PTV},
\begin{equation*}
\left\{\begin{aligned}
&\rho_t+\mbox{div}\,(\rho\nabla p)=\rho\Phi(p,c)\\
&c_t-\Delta c=-\rho\Psi (p,c)\\
&c(x,t)\to c_B>0\quad\mbox{ as }|x|\to\infty,
\end{aligned}\right.
\end{equation*}
as $m\to\infty$, where $p=k\rho^{m-1}$ for some constant $k>0$ and $\Phi$, $\Psi$, are smooth functions that satisfy some structural conditions. 

In this paper we will study the singular limit of solutions $u_{m,p}$ of the generalized Burgers equation with absorption,
\begin{equation}\label{burgers-absorption-eqn}
\left\{\begin{aligned}
&u_t+(u^m)_x=-u^p \qquad  \mbox{in }\R\times(0,\infty)\\
&u(x,0)=u_0(x)\ge 0 \quad \mbox{ in }\R
\end{aligned}\right.
\end{equation}
when either $m\to\infty$ or $p\to\infty$. We will prove that under some mild conditions on the initial data $u_0$, as $m\to\infty$ or $p\to\infty$, the singular limit of solutions of \eqref{burgers-absorption-eqn} exists. 

More precisely we will prove the following three results.

\begin{thm}\label{main-thm1}
Let $0\le u_0\in L^1(\R)\cap L^{\infty}(\R)$. For any $p>1$, $m>1$, let $u_{m,p}$ be the solution of \eqref{burgers-absorption-eqn} in $\R\times(0,\infty)$ given by Lemma \ref{thm-existence-of-general-solution}. Then as $m\to\infty$, $u_{m,p}$ converges in $C\left([t_0,T];L_{loc}^1(\R)\right)$ for any $T>t_0>0$ to some function $u_{\infty ,p}$, $0\le u_{\infty,p}\le 1$, which satisfies
\begin{equation}\label{p-power-ode}
u_t=-u^p \quad \mbox{ in }\,\mathcal{D}'(\R\times(0,\infty))
\end{equation}
with initial value $u_{\infty}^0(x)$, $0\leq u_{\infty}^0\leq 1$, that satisfies
\begin{equation}\label{p-limit-initial-value}
u_{\infty}^0(x)+\psi_x(x)=u_0(x)  \quad \mbox{ in }\mathcal{D}'(\R)
\end{equation}
for some function $0\le\psi\in L^1(\R)\cap L^{\infty}(\R)$ satisfying 
\begin{equation}\label{psi=0-eqn}
\psi(x)=0\quad\mbox{ a.e. }x\in \{x\in\R:u_{\infty}^0(x)<1\}.
\end{equation}
\end{thm}

\begin{thm}\label{main-thm2}
Let $0\le u_0\in L^{\infty}(\R)$. For any $p>1$, $m>1$, let $u_{m,p}$ be the solution of \eqref{burgers-absorption-eqn} in $\R\times(0,\infty)$ given by Lemma \ref{thm-existence-of-general-solution}. Then as $p\to\infty$, $u_{m,p}$ converges in $C([t_0,T];L^1_{loc}(\R))$  for any $T>t_0>0$ to the solution $u_{m,\infty}$ of the equation,
\begin{equation}\label{burgers-eqn}
\left\{\begin{aligned}
&u_t+(u^m)_x=0 \qquad\qquad\qquad\, \mbox{ in }\R\times(0,\infty)\\
&u(x,0)=\min (u_0(x),1)\qquad \quad\mbox{ in }\R.
\end{aligned}\right.
\end{equation}
\end{thm}

\begin{thm}\label{main-thm3}
Let $0\leq u_0\in L^{1}(\R)\cap L^{\infty}(\R)$. For any $p>1$, $m>1$, let $u_{m,p}$ be the solution of \eqref{burgers-absorption-eqn} and let $u_{\infty,p}$, $u^0_{\infty}$, $u_{m,\infty}$, be given by Theorem \ref{main-thm1} and Theorem \ref{main-thm2} respectively. Then the following holds:
\begin{itemize}
\item[(i)] as $m\to\infty$, $u_{m,\infty}$ converges  in $L_{loc}^1(\R\times (0,\infty))$ to some function $v_1$ on $\R$, $0\le v_1\le 1$, which satisfies
\begin{equation}\label{eq-condition-of-v-1-infty-indty}
v_1(x)+\psi_1(x)_x=\min\left\{u_0(x),1\right\} \qquad \mbox{in $\mathcal{D}'(\R)$}
\end{equation}
for some function $0\le\psi_1\in L^1(\R)\cap L^{\infty}(\R)$ such that $\psi_1(x)=0$ for a.e. $x\in \left\{x:v_1(x)<1\right\}$.
\item[(ii)] as $p\to\infty$, $u_{\infty,p}$ converges weakly in $L^1(\R\times (0,\infty))$ to $u_{\infty}^0$.
\end{itemize}
\end{thm}

Note that as a consequence of Theorem \ref{main-thm3} in general we have
\begin{equation*}
\lim_{p\to\infty}\lim_{m\to\infty}u_{m,p}\ne\lim_{m\to\infty}\lim_{p\to\infty}u_{m,p}.
\end{equation*}

The plan of the paper is as follows. We will prove Theorem \ref{main-thm1} and Theorem \ref{main-thm2} in section two and section three respectively. In section four we will prove Theorem \ref{main-thm3}.

We start with some definitions. We will use the definition of solution in \cite{K} for \eqref{burgers-absorption-eqn}. For any $\varphi\in C^{1}([0,\infty))$, we say that a function $0\leq u\in L^{\infty}(\R\times(0,\infty))$ is a solution of 
\begin{equation}\label{eq-cases-main-problem-general}
\begin{cases}
\begin{aligned}
&u_t+(u^m)_x=\varphi(u) \quad\,  \mbox{ in }\R\times(0,\infty)\\
&u(x,0)=u_0(x)\ge 0 \quad \mbox{in }\R
\end{aligned}
\end{cases}
\end{equation}
if it satisfies the following two conditions:
\begin{enumerate}
\item[\textbf{(i)}] for any $k\in\R$ and $0\le\eta\in C^{\infty}_0(\R\times(0,\infty))$,
\begin{equation*}
\int_{0}^T\int_{\R}\{|u(x,t)-k|\eta_t+|u(x,t)^m-k^m|\eta_x+\mbox{sign}\,(u(x,t)-k)\varphi(u)\eta\}\,dxdt\geq 0;
\end{equation*} 
\item[\textbf{(ii)}] there exists a set $\mathcal{E}$ of measure zero on $[0,\infty)$ such that for any $t\in[0,\infty)\bs\mathcal{E}$ the function $u(x,t)$ is defined almost everywhere in $\R$ and
\begin{equation*}
\lim_{\begin{subarray}{c}t\in[0,\infty)\bs\mathcal{E}\\t\to 0\end{subarray}}\int_{B_r}\left|u(x,t)-u_0(x)\right|\,dx=0
\end{equation*}
holds for any ball $B_r=\{x\in\R:|x|\leq r\}$. 
\end{enumerate}

As observed by \cite{K} for any solution $u$ of \eqref{eq-cases-main-problem-general}, $u$ satisfies
\begin{equation}\label{eq-equation-for-solution-maybe-weak-concept}
\int_{0}^{\infty}\int_{-\infty}^{\infty}u\eta_t\,dxdt+\int_{0}^{\infty}\int_{-\infty}^{\infty}u^m\eta_x\,dxdt=-\int_{0}^{\infty}\int_{-\infty}^{\infty}\varphi(u)\eta\,dxdt
\end{equation}
for any $0\leq\eta\in C_0^{\infty}(\R\times(0,\infty))$.
Note that by the result and the proof of \cite{K} we have the following two results.

\begin{lemma}\label{thm-existence-of-general-solution}
Let $m>1$, $p>1$ and $0\le u_0\in L^{\infty}(\R)$. Then there exist unique solutions $u_{m,p}$, $v_m$, of  \eqref{burgers-absorption-eqn} and 
\begin{equation}\label{burgers-eqn0}
\left\{\begin{aligned}
u_t+(u^m)_x&=0\qquad\,\,\,\,\mbox{ in }\R\times(0,\infty)\\
u(x,0)=&u_0(x)\qquad\mbox{ in }\R.
\end{aligned}\right.
\end{equation}
respectively which satisfy 
\begin{equation}\label{L-infty-bd}
0\le u_{m,p}\le v_m\le\left\|u_0\right\|_{L^{\infty}(\R)}\quad\mbox{ in }\R\times (0,\infty).
\end{equation}
If $0\le u_0\in L^{\infty}(\R)\cap L^1(\R)$, then
\begin{equation}\label{eq-relation-between-solutions-u-v-with-initial-data-u-0-in-L-1}
\int_{\R}u_{m,p}(x,t)\,dx\le\int_{\R}v_m(x,t)\,dx=\int_{\R}u_0(x)\,dx \qquad \forall t>0.
\end{equation}
\end{lemma}

\begin{lemma}\label{lemma-L-1-contraction-of-soutions}
Let $m>1$, $p>1$, $0\le u_0\in L^{\infty}(\R)$, and $u_{m,p}$ be the unique solution of \eqref{burgers-absorption-eqn} in $\R\times(0,\infty)$. Then for any $R>1$ and $T>t_0>0$ there exists a monotone increasing function $\omega_R\in C([0,\infty))$, $\omega_R(0)=0$, depending only on $R$, $\left\|u_0\right\|_{L^{\infty}}$, $m\left\|u_{m,p}\right\|^{m-1}_{L^{\infty}(\R\times(t_0,T])}$ and $p\left\|u_{m,p}\right\|^{p-1}_{L^{\infty}(\R\times(t_0,T])}$ such that
\begin{equation}\label{eq-similar-to-L-1-contraction-for-precompactness-1}
\int_{|x|<R}\left|u_{m,p}(x+x_0,t)-u_{m,p}(x,t)\right|\,dx\leq \omega_R(|x_0|) \qquad  \forall |x_0|\leq 1,\,\, t_0\leq t\leq T
\end{equation}
and
\begin{equation}\label{eq-similar-to-L-1-contraction-for-precompactness-2}
\int_{|x|<R}\left|u_{m,p}(x,t_1)-u_{m,p}(x,t_2)\right|\,dx\leq \omega_R(|t_1-t_2|) \qquad  \forall t_1,t_2\in[t_0,T].
\end{equation}
\end{lemma}

By Theorem 1 of \cite{K} and Lemma \ref{thm-existence-of-general-solution} we have the following result.

\begin{lemma}\label{lem-L-1-contraction-of-solutions-u-1--and-u-2-2349}
Let $m>1$, $p>1$, and  $u_{0,1}$, $u_{0,2}\in L^{\infty}(\R)$ be non-negative functions on $\R$. Suppose $u_1$, $u_2$, are the solutions of \eqref{burgers-absorption-eqn} in $\R\times(0,\infty)$ with initial value $u_0=u_{0,1}, u_{0,2}$, respectively. Let 
\begin{equation*}
N=\max\left\{m\|u_{0,1}\|_{L^{\infty}(\R)}^{m-1},m\|u_{0,2}\|_{L^{\infty}(\R)}^{m-1}\right\}. 
\end{equation*}
Then 
\begin{equation}\label{eq-L-1-contraction-fo-second-order-PDE-23-with-differences-of-two-sols}
\|u_1(\cdot,t)-u_{2}(\cdot,t)\|_{L^1(B_{R-Nt})}\le\|u_{0,1}-u_{0,2}\|_{L^1(B_R)} \qquad \forall 0<t<R/N,R>0,p>1.
\end{equation}
\end{lemma}

We will now assume that $0\le u_0\in L^{\infty}(\R)$ and let $u_{m,p}$, $v_m$, be the solutions of \eqref{burgers-absorption-eqn} and\eqref{burgers-eqn0} respectively for the rest of the paper. For any $x_0\in \R$ and $R>0$, we let
$B_R(x_0)=\{x\in\R:|x-x_0|<R\}$ and $B_R=B_R(0)$.

\section{Singular limit as $m\to\infty$}\label{section-m-infty}
\setcounter{equation}{0}
\setcounter{thm}{0}

In this section we will prove Theorem \ref{main-thm1}. For fixed $p>1$, we will write $u_m:=u_{m,p}$ for any $m>1$. We will also assume that $0\le u_0\in L^{\infty}(\R)\cap L^1(\R)$ and let 
\begin{equation*}
\psi_m(x,t)=\int_0^{t}u_m(x,\tau)^m\,d\tau
\end{equation*}
in this section. Let $\left\{m_i\right\}_{i=1}^{\infty}\in\Z^+$ be a sequence such that $m_i\to\infty$ as $i\to\infty$. By \eqref{L-infty-bd} and the result on P. 64 of \cite{X},
\begin{equation}\label{eq-upper-bound-of-u-m-by-v-m-the-solution-of-homogeneous-case}
0\le u_m(x,t)^m\le v_m(x,t)^m\leq \frac{2\left\|u_0\right\|_{L^1(\R)}}{(m-1)t}, \qquad \mbox{a.e. $(x,t)\in\R\times(0,\infty)$}.
\end{equation}
Then
\begin{equation*}
m\left(u_m\right)^{m-1}\leq m\left(\frac{2\left\|u_0\right\|_{L^1(\R)}}{(m-1)t}\right)^{\frac{m-1}{m}}=\left(\frac{m}{m-1}\right)(m-1)^{\frac{1}{m}}\left(\frac{2\left\|u_0\right\|_{L^1(\R)}}{t}\right)^{\frac{m-1}{m}}.
\end{equation*}
Hence for any $t_0>0$ there exists a constant $M_{t_0}>0$ such that 
\begin{equation*}
m\|u_m\|_{L^{\infty}(\R\times (t_0,\infty))}^{m-1}\le M_{t_0}\quad \forall m\ge 2.
\end{equation*}
Thus for any $R>1$ and $T>t_0>0$, we can choose the function $\omega_{R}$ in Lemma \ref{lemma-L-1-contraction-of-soutions} to be independent  of $m\geq 2$. Hence, by \eqref{L-infty-bd}, \eqref{eq-similar-to-L-1-contraction-for-precompactness-1} and \eqref{eq-similar-to-L-1-contraction-for-precompactness-2}, the sequence $\left\{u_{m_i}\right\}_{i=1}^{\infty}$ is equi-continuous in $C\left([t_0,T];L_{loc}^1(\R)\right)$ for any $T>t_0>0$. Thus by \eqref{eq-upper-bound-of-u-m-by-v-m-the-solution-of-homogeneous-case}, the Ascoli theorem and a diagonalization argument the sequence $\{u_{m_i}\}_{i=1}$ has a subsequence which we may assume without loss of generality to be the sequence itself that converges in $C\left([t_0,T];L_{loc}^1(\R)\right)$ for any $T>t_0>0$ to some function $u_{\infty, p}\in C\left((0,\infty);L_{loc}^1(\R)\right)$, $0\le u_{\infty,p}\le 1$, as $i\to\infty$. When there is no ambiguity we will drop the subscript p and write $u_{\infty}$ for $u_{\infty, p}$.

\begin{lemma}\label{lem-convergence-of-u-m-to-u-infty-with-uniqueness}
$u_{\infty}$ satisfies \eqref{p-power-ode}. 
\end{lemma} 
\begin{proof}
By \eqref{eq-upper-bound-of-u-m-by-v-m-the-solution-of-homogeneous-case}, $\left(u_m\right)^m\to 0$ uniformly on $\R\times[T_0,\infty)$ for any fixed $T_0>0$ as $m\to \infty$. Putting $u=u_m$, $\varphi(u)=-(u_m)^p$, $m=m_i$, in \eqref{eq-equation-for-solution-maybe-weak-concept} and letting $i\to\infty$, we get
\begin{equation*}
\int_{0}^{\infty}\int_{-\infty}^{\infty}u_{\infty}\eta_t\,dxdt=\int_{0}^{\infty}\int_{-\infty}^{\infty}\left(u_{\infty}\right)^p\eta\,dxdt \qquad \forall 0\le\eta\in C_0^{\infty}(\R\times(0,\infty)).
\end{equation*}
and \eqref{p-power-ode} follows.
\end{proof}

\begin{lemma}\label{lem-precompactness-of-int-u-sub-m-to-m-over-o-to-t}
For any $T>0$ the sequence of functions $\{\psi_m(x,t)\}_{m>p}$ is equi-continuous in $C([0,T);L^1(\R))$.
\end{lemma}
\begin{proof}
We will use a modification of the technique of \cite{X} to prove the lemma. We first extend $u_m$ to a function on $\R^2$ by letting $u_m(x,t)=0$ for all $t<0$, $x\in\R$. Since $u_m$ satisfies \eqref{eq-equation-for-solution-maybe-weak-concept} with $\phi (u)=-(u_m)^p$, by \eqref{eq-equation-for-solution-maybe-weak-concept} and  an approximation argument,
\begin{equation}\label{eq-integration-multiplying-test-function-and-by-parts--1}
\int_{-\infty}^{\infty}\int_{-\infty}^{\infty}u_m\eta_t\,dxdt+\int_{\infty}^{\infty}\int_{-\infty}^{\infty}(u_m)^m\eta_x\,dxdt+\int_{-\infty}^{\infty}u_0\eta\,dx=\int_{-\infty}^{\infty}\int_{-\infty}^{\infty}(u_m)^p\eta\,dxdt\quad\forall 0\le\eta\in C_0^{\infty}(\R^2).
\end{equation}
We choose $\Phi\in C_0^{\infty}(\R^2)$, $0\leq\Phi\leq 1$, such that $\int_{\R}\int_{\R}\Phi\,dxdy=1$ 
and let $J_{\3}(x,t)=\frac{1}{\3}\Phi\left(\frac{x}{\3},\frac{t}{\3}\right)$ for any $\3>0$. 
Putting $\eta(x,t)=J_{\3}(\xi-x,\tau-t)$ in \eqref{eq-integration-multiplying-test-function-and-by-parts--1},
\begin{equation}\label{eq-integration-substituting-mollifier}
-\left(A_{m,\3}(\xi,\tau)\right)_{\tau}-\left(B_{m,\3}(\xi,\tau)\right)_{\xi}+\int_{-\infty}^{\infty}J_{\3}(\xi-x,\tau)u_0(x)\,dx=C_{m,\3}(\xi,\tau)
\end{equation}
where
\begin{equation*}
A_{m,\3}(\xi,\tau)=\int_{-\infty}^{\infty}\int_{-\infty}^{\infty}u_m(x,t)J_{\3}(\xi-x,\tau-t)\,dxdt,
\end{equation*}
\begin{equation*}
B_{m,\3}(\xi,\tau)=\int_{-\infty}^{\infty}\int_{-\infty}^{\infty}\left(u_m\right)^m(x,t)J_{\3}(\xi-x,\tau-t)\,dxdt,
\end{equation*}
and
\begin{equation*}
C_{m,\3}(\xi,\tau)=\int_{-\infty}^{\infty}\int_{-\infty}^{\infty}\left(u_m\right)^p(x,t)J_{\3}(\xi-x,\tau-t)\,dxdt.
\end{equation*}
Integrating  \eqref{eq-integration-substituting-mollifier} first with respect to $\xi$ over $(x,x+h)$, $h>0$, and then with respect to $\tau$ over $(\sigma,t)$, $t>\sigma>0$, 
\begin{equation}\label{intermediate-step1}
\begin{aligned}
&\int_{\sigma}^{t}B_{m,\3}(x+h,\tau)\,d\tau-\int_{\sigma}^{t}B_{m,\3}(x,\tau)\,d\tau+\int_{\sigma}^{t}\int_{x}^{x+h}C_{m,\3}(\xi,\tau)\,d\xi d\tau\\
&\qquad \qquad =-\int_{x}^{x+h}\left(A_{m,\3}(\xi,t)-A_{m,\3}(\xi,\sigma)\right)\,d\xi+\int_{-\infty}^{\infty}\int_{x}^{x+h}\int_{\sigma}^{t}J_{\3}(\xi-z,\tau)u_0(z)\,d\tau d\xi dz.
\end{aligned}
\end{equation}
Similar to the proof on P.63--64 of \cite{X}, letting $\3\to 0$ in \eqref{intermediate-step1},
\begin{equation*}
\begin{aligned}
&\int_{\sigma}^{t}\left[u_{m}(x+h,\tau)^m-u_{m}(x,\tau)^m\right]\,d\tau\\
&\qquad \qquad =-\int_{x}^{x+h}\left(u_{m}(\xi,t)-u_{m}(\xi,\sigma)\right)\,d\xi-\int_{\sigma}^{t}\int_{x}^{x+h}u_{m}(\xi,\tau)^p\,d\xi d\tau \quad \mbox{a.e. $(x,t)\in\R\times(0,\infty)$}.
\end{aligned}
\end{equation*}
Letting $\sigma\to 0$,
\begin{equation}\label{eq-aligned-space-difference-of-psi-34}
\psi_{m}(x+h,t)-\psi_{m}(x,t)=\int_{x}^{x+h}u_{0}(\xi)\,d\xi-\int_{x}^{x+h}u_{m}(\xi,t)\,d\xi
-\int_{0}^{t}\int_{x}^{x+h}u_{m}(\xi,\tau)^p\,d\xi d\tau
\end{equation}
for a.e. $(x,t)\in\R\times(0,\infty)$.
By \eqref{eq-relation-between-solutions-u-v-with-initial-data-u-0-in-L-1} and \eqref{eq-upper-bound-of-u-m-by-v-m-the-solution-of-homogeneous-case},
\begin{align}
\int_{\R}\int_{0}^{t}\int_{x}^{x+h}u_m^p(\xi,\tau)\,d\xi d\tau dx&\leq h\int_{0}^{t}\int_{\R}u_m^p(\xi,\tau)\,d\xi d\tau\notag\\
&\leq \left(\frac{2\left\|u_0\right\|_{L^1(\R)}}{m-1}\right)^{\frac{p-1}{m}}h\int_{0}^{t}\left(\frac{1}{\tau^{\frac{p-1}{m}}}\int_{\R}u_m\,d\xi\right)d\tau\notag\\
&\leq h\left(\frac{1}{1-\frac{p-1}{m}}\right)\left(\frac{2\left\|u_0\right\|_{L^1(\R)}}{m-1}\right)^{\frac{p-1}{m}}\left\|u_0\right\|_{L^1(\R)}t^{1-\frac{p-1}{m}} \qquad \forall m>p-1.
\label{eq-aligned-control-of-term-with-u-^-p-after-some-calculus}
\end{align}
Hence, by  \eqref{eq-relation-between-solutions-u-v-with-initial-data-u-0-in-L-1}, \eqref{eq-aligned-space-difference-of-psi-34} and \eqref{eq-aligned-control-of-term-with-u-^-p-after-some-calculus},
\begin{equation}\label{eq-space-difference-of-psi}
\begin{aligned}
&\int_{\R}\left|\psi_{m}(x+h,t)-\psi_{m}(x,t)\right|\,dx\\
&\qquad \qquad \le h\left[2+\left(\frac{1}{1-\frac{p-1}{m}}\right)\left(\frac{2\left\|u_0\right\|_{L^1(\R)}}{m-1}\right)^{\frac{p-1}{m}}t^{1-\frac{p-1}{m}}\right]\left\|u_0\right\|_{L^1(\R)}
\quad\mbox{ a.e. }t>0,\quad \forall m>p-1, \,\,h>0.
\end{aligned}
\end{equation}
By \eqref{eq-relation-between-solutions-u-v-with-initial-data-u-0-in-L-1} and \eqref{eq-upper-bound-of-u-m-by-v-m-the-solution-of-homogeneous-case},
\begin{align}
\int_{\R}\left|\psi_m(x,t+h)-\psi_m(x,t)\right|\,dx&=\int_t^{t+h}\int_{\R}u_m(x,\tau)^m\,dxd\tau\notag\\
&\leq\frac{m}{m-1}(m-1)^{\frac{1}{m}}2^{1-\frac{1}{m}}\left\|u_0\right\|_{L^1(\R)}^{2-\frac{1}{m}}((t+h)^{\frac{1}{m}}-t^{\frac{1}{m}})\label{eq-time-difference-of-psi}
\end{align}
for a.e $t>0$, $h>0$, and any $m>p-1$.
By \eqref{L-infty-bd} and \eqref{eq-upper-bound-of-u-m-by-v-m-the-solution-of-homogeneous-case},
\begin{equation}\label{eq-boundedness-of-psi}
0\le \psi_m(x,t)=\int_{0}^{t}u_m(x,\tau)^m\,d\tau\leq \frac{m}{m-1}(m-1)^{\frac{1}{m}}(2\left\|u_0\right\|_{L^1(\R)})^{\frac{m-1}{m}}\|u_0\|_{L^{\infty}(\R)}t^{\frac{1}{m}}\quad\mbox{ a.e. }x\in\R,t>0.
\end{equation}
By \eqref{eq-space-difference-of-psi}, \eqref{eq-time-difference-of-psi} and \eqref{eq-boundedness-of-psi},  the lemma follows.
\end{proof}

By \eqref{eq-boundedness-of-psi} and Lemma \ref{lem-precompactness-of-int-u-sub-m-to-m-over-o-to-t} the sequence $\{\psi_{m_i}\}_{i=1}^{\infty}$ has a subsequence which we may assume without loss of generality to the sequence itself such that $\psi_{m_i}$ converges in $C([0,T];L^1(\R))$ to some function $0\le\psi\in C([0,\infty);L^1(\R))\cap L^{\infty}(\R\times (0,T))$ for any $T>0$ as $i\to\infty$.

By an argument similar to the proof of Theorem 1.2 of \cite{X}, the following result holds.

\begin{prop}\label{eq-proposition-independent-of-t-in-psi}
The function $\psi$ is independent of $t$.
\end{prop}

By Proposition \ref{eq-proposition-independent-of-t-in-psi}, $0\le\psi\in L^1(\R)\cap L^{\infty}(\R)$.
We are now ready for the proof of Theorem \ref{main-thm1}.

\begin{proof}[\textbf{Proof of Theorem \ref{main-thm1}}]
By the previous arguments it remains to prove the uniqueness of $u_{\infty}$. Let $\eta\in C^{\infty}_0(\R^2)$. We first claim that 
\begin{equation}\label{eq-aligned-claim-for-intetgration-by-oparts-of-psi}
\lim_{i\to\infty}\int_{0}^{\infty}\int_{\R}\left(u_{m_i}\right)^{m_i}\eta_x\,dxdt=\int_{\R}\psi(x)\eta_{x}(x,0)\,dx.
\end{equation}
To prove the claim we choose $R_0>0$, $T_0>0$, such that  
$\text{supp }\,\eta\subset[-R_0,R_0]\times[-T_0,T_0]$.
Then
\begin{align}
\int_{0}^{\infty}\int_{\R}\left(u_{m_i}\right)^{m_i}\eta_x\,dxdt&=\int_{\delta}^{T_0}\int_{\R}\left(u_{m_i}\right)^{m_i}\eta_x\,dxdt+\int_{0}^{\delta}\int_{\R}\left(u_{m_i}\right)^{m_i}(\eta_x(x,t)-\eta_x(x,0))\,dxdt\notag\\
&\quad +\int_{\R}\eta_x(x,0)\left(\int_{0}^{\delta}u_{m_i}(x,t)^{m_i}\,dt\right)\,dx\notag\\
&=I_1+I_2+I_3 \qquad \qquad \forall 0<\delta<T_0.\label{eq-aligned-split-for-intetgration-by-oparts-of-osi}
\end{align}
By \eqref{eq-upper-bound-of-u-m-by-v-m-the-solution-of-homogeneous-case}, 
\begin{equation}\label{eq-control-of-I-1-of-sprit-of-integration-for-psi}
I_1 \to 0\qquad  \mbox{as $i\to\infty$}.
\end{equation}
By the mean value theorem, for any $x\in\R$, $t>0$, there exists a constant $t_x\in(0,t)$ such that
\begin{equation*}
\eta_x(x,t)-\eta_x(x,0)=t\eta_{xt}(x,t_x).
\end{equation*}
Then by \eqref{eq-upper-bound-of-u-m-by-v-m-the-solution-of-homogeneous-case},
\begin{align}\label{eq-control-of-I-2-of-sprit-of-integration-for-psi}
|I_2|&=\left|\int_{0}^{\delta}\int_{\R}\left(u_{m_i}\right)^{m_i}\eta_{xt}(x,t_x)\,t\,dxdt\right|
\le\int_0^{\delta}\int_{-R_0}^{R_0}\left(\frac{2\|u_0\|_{L^1(\R)}\|\eta_{xt}\|_{L^{\infty}(\R)}}{m_i-1}\right)\,dxdt\notag\\
&\le\frac{4\delta R_0\|u_0\|_{L^1(\R)}\|\eta_{xt}\|_{L^{\infty}(\R)}}{m_i-1}\to 0\qquad\mbox{as $i\to\infty$}.
\end{align}
By Proposition \ref{eq-proposition-independent-of-t-in-psi},
\begin{equation}\label{eq-control-of-I-3-of-sprit-of-integration-for-psi}
\int_0^{\delta}u_{m_i}(x,t)^{m_i}\,dt\to\psi(x) \qquad \mbox{in $L^1(\R)$} \qquad  \mbox{as $i\to\infty$}.
\end{equation}
Letting $i\to\infty$ in \eqref{eq-aligned-split-for-intetgration-by-oparts-of-osi}, by \eqref{eq-control-of-I-1-of-sprit-of-integration-for-psi}, \eqref{eq-control-of-I-2-of-sprit-of-integration-for-psi} and \eqref{eq-control-of-I-3-of-sprit-of-integration-for-psi}, the claim \eqref{eq-aligned-claim-for-intetgration-by-oparts-of-psi} follows.

Since $u_{\infty}$ satisfies \eqref{p-power-ode},  $u_{\infty}(x,t)$ is monotone decreasing in $t>0$. Hence 
\begin{equation*}
u_{\infty}^0(x):=u_{\infty}(x,0)=\lim_{t\to 0}u_{\infty}(x,t) \qquad \mbox{exists.}
\end{equation*}
Putting $m=m_i$ in \eqref{eq-integration-multiplying-test-function-and-by-parts--1} and letting $i\to\infty$,
\begin{equation}\label{eq-for-initial-of-u-sub-infty-at-t-=-0-after-limit}
\int_{0}^{\infty}\int_{\R}u_{\infty}\eta_t\,dxdt+\int_{\R}u_0(x)\eta(x,0)\,dx+\int_{\R}\psi(x)\eta_x(x,0)\,dx=\int_{0}^{\infty}\int_{\R}\left(u_{\infty}\right)^p\eta\,dxdt\qquad\forall 0\le\eta\in C_0^{\infty}(\R^2).
\end{equation}
We now choose $\phi\in C^{\infty}(\R)$, $0\le\phi\le 1$, such that $\phi(r)=0$ for all  $r\le-1$ and 
$\phi(r)=1$ for all $r\ge 0$ and let $\phi_{\3}(r)=\phi(r/\3)$ for any $r\in\R$ and $\3>0$. For any $\eta\in C_{0}^{\infty}(\R)$ and $t_0>0$, by replacing $\eta$ by $\phi_{\3}(t)\phi_{\3}(t_0-t)\eta(x)$ in \eqref{eq-for-initial-of-u-sub-infty-at-t-=-0-after-limit} and letting $\3\to 0$, we have
\begin{equation}\label{eq-for-initial-of-u-sub-infty-at-t-=-0-after-epsilon-limit}
-\int_{\R}u_{\infty}(x,t_0)\eta(x)\,dx+\int_{\R}u_0(x)\eta(x)\,dx+\int_{\R}\psi(x)\eta_x(x)\,dx=\int_{0}^{t_0}\int_{\R}u_{\infty}^p\eta\,dxdt.
\end{equation}
Letting $t_0\to 0$ in \eqref{eq-for-initial-of-u-sub-infty-at-t-=-0-after-epsilon-limit}, by the monotone convergence theorem,
\begin{equation*}
-\int_{\R}u_{\infty}^0(x)\eta(x)\,dx+\int_{\R}u_0(x)\eta(x)\,dx+\int_{\R}\psi(x)\eta_x(x)\,dx=0\quad\forall\eta\in C_0^{\infty}(\R)
\end{equation*}
and \eqref{p-limit-initial-value} holds.
We are now going to prove \eqref{psi=0-eqn}. For any $k>1$ let $\eta_k(x)=\phi(x+k)\phi(k-x)$. Then
$0\le\phi_k\le 1$, $\eta_k(x)=1$ for any  $|x|\le k$ and $\eta_{k}(x)=0$ for any  $|x|\ge k+1$. 
By \eqref{p-limit-initial-value} there exists a constant $C>0$ such that
\begin{equation}\label{eq-for-energy-conservation-of-u-infty-zro}
\left|\int_{\R}u_{\infty}^0\eta_k\,dx-\int_{\R}u_0\eta_k\,dx\right|\le C\int_{k\leq |x|\leq k+1}\psi\,dx\quad\forall k>1.
\end{equation}
Since $\psi\in L^1(\R)$, letting $k\to\infty$ in \eqref{eq-for-energy-conservation-of-u-infty-zro}, 
\begin{equation}\label{eq-energy-conservation-of-u-infty-zzero}
\int_{\R}u_{\infty}^0\,dx=\int_{\R}u_0\,dx.
\end{equation}
We now recall that by the result of \cite{X}, 
\begin{equation*}
v_m(x,t)\to v_{\infty}(x)\quad\mbox{ and }\quad\int_0^{t}v_m(x,t)^m\,dt\to \4{\psi}(x)\qquad \mbox{as }
m\to \infty\quad\mbox{ in }L^1_{loc}(\R\times(0,\infty))
\end{equation*}
for some functions $v_{\infty}(x)$, $\4{\psi}(x)$, which satisfy
\begin{equation}\label{eq-aligned-contition-of-v-infty-boundedness-and-envergy-conservation}
\begin{aligned}
 0\leq v_{\infty}(x)\leq 1, \qquad \int_{\R}v_{\infty}\,dx=\int_{\R}u_0\,dx,\qquad 0\le\4{\psi}(x)\in L^1(\R)\cap 
 L^{\infty}(\R), 
\end{aligned}
\end{equation}
and
\begin{equation}\label{eq-aligned-eq-for-v-infty-with-tilde-psi}
\begin{aligned}
v_{\infty}(x)+\4{\psi}_x(x)=u_0(x) \quad \mbox{in $\mathcal{D}'(\R)$}, 
\end{aligned}
\end{equation}
with 
\begin{equation}\label{eq-energy-conservation-of-v-infty-zzero}
\4{\psi}(x)=0 \quad\mbox{ a.e. }x\in\{x\in\R:v_{\infty}(x)<1\}.
\end{equation}
Since $u_m(x,t)\leq v_m(x,t)$, we have 
\begin{equation}\label{psi-and-psi-tilde-compare}
0\leq \psi\leq\4{\psi}
\end{equation} 
and
\begin{equation}\label{eq-comparison-between-v-infty-zero-and-u-infty}
u_{\infty}(x,t)\leq v_{\infty}(x)\qquad \Rightarrow \qquad u_{\infty}^0(x)\leq v_{\infty}(x).
\end{equation}
By \eqref{eq-energy-conservation-of-u-infty-zzero} and \eqref{eq-aligned-contition-of-v-infty-boundedness-and-envergy-conservation}, 
\begin{equation}\label{eq-no-difference-between-u-infty-and-v-infty-in-L-1}
\int_{\R}u_{\infty}^0(x)\,dx=\int_{\R}v_{\infty}(x)\,dx.
\end{equation}
By \eqref{eq-comparison-between-v-infty-zero-and-u-infty} and \eqref{eq-no-difference-between-u-infty-and-v-infty-in-L-1},
\begin{equation}\label{eq-comparition-between-u-infty-0-and-v-infty-not-L-1-spoace}
u_{\infty}^0(x)=v_{\infty}(x) \quad\mbox{ a.e. }x\in\R.
\end{equation}
By \eqref{eq-energy-conservation-of-v-infty-zzero}, \eqref{psi-and-psi-tilde-compare}  and \eqref{eq-comparition-between-u-infty-0-and-v-infty-not-L-1-spoace}, we get \eqref{psi=0-eqn}.  By the discussion on P.70 of \cite{X}, $u_{\infty}^0$ is uniquely determined by \eqref{p-limit-initial-value} and \eqref{psi=0-eqn}. Since $u_{\infty}$ satisfies \eqref{p-power-ode} with initial value $u_{\infty}^0$, the function $u_{\infty}$ is unique. Since the sequence $\{m_i\}_{i=1}^{\infty}$ is arbitrary, $u_{m}$ converges to $u_{\infty}$ in $C([t_0,T];L_{loc}^1(\R))$ for any $T>t_0>0$ as $m\to\infty$ and  Theorem \ref{main-thm1} follows.
\end{proof}

\section{Singular limit as $p\to\infty$}\label{section-p-infty}
\setcounter{equation}{0}
\setcounter{thm}{0}

In this section we will prove Theorem \ref{main-thm2}. We will fix $m>1$ and write $w_p:=u_{m,p}$ for any $p>1$. 
We will assume that $0\le u_0\in L^{\infty}(\R)$ in this section.

\begin{lemma}\label{lem-upper-bound-of-u-m-p-in-section-4}
$w_p$ satisfies 
\begin{equation}\label{eq-bound-of-characteristic-solutions-for-each-m}
w_p(x,t)\leq \frac{1}{\left((p-1)t+\left\|u_0\right\|^{1-p}_{L^{\infty}}\right)^{\frac{1}{p-1}}} \qquad \mbox{a.e.}\,\, (x,t)\in\R\times(0,\infty)\quad\forall p>1. 
\end{equation}
\end{lemma}
\begin{proof}
By direct computation, the function
\begin{equation*}
h(x,t)=\frac{1}{\left((p-1)t+\left\|u_0\right\|_{L^{\infty}(\R)}^{1-p}\right)^{\frac{1}{p-1}}}
\end{equation*}
satisfies
\begin{equation*}
\begin{aligned}
\begin{cases}
&u_t+\left(u^m\right)_x=\3 u_{xx}-u^p \qquad \mbox{in $\R\times(0,\infty)$}\\
&\quad u(x,0)=\left\|u_0\right\|_{L^{\infty}(\R)}. \qquad \mbox{ in }\R
\end{cases}
\end{aligned}
\end{equation*}
for any $\3>0$. Let $u_{m,p}^{\3}(x,t)$ be the solution of  the problem 
\begin{equation}\label{epsilon-approx-eqn}
\left\{\begin{aligned}
u_t+(u^m)_x=&\3 u_{xx}-u^p \quad \mbox{in $\R\times(0,\infty)$}\\
u(x,0)=&u_0(x) \qquad\,\,\, \mbox{ in }\R.
\end{aligned}\right.
\end{equation}
By the construction of solution in \cite{K}, $u_{m,p}^{\3}$ converges almost everywhere in $\R\times(0,\infty)$ to $w_p$ as $\3\to 0^+$. By the maximum principle for parabolic equation,
\begin{align*}
&u_{m,p}^{\3}(x,t)\leq h(x,t) \quad\forall (x,t)\in\R\times(0,\infty)\\
\Rightarrow \quad &w_p(x,t)\leq h(x,t) \qquad \mbox{a.e.}\,\,(x,t)\in\R\times(0,\infty) \qquad \mbox{as $\3\to 0^+$}
\end{align*}
and the lemma follows.
\end{proof}

Let $\left\{p_i\right\}_{i=1}^{\infty}\subset\R^+$, $p_i>2$, $\forall i,\cdots$, be such that $p_i\to\infty$ as $i\to\infty$. By \eqref{eq-bound-of-characteristic-solutions-for-each-m} for any $t_0>0$,
\begin{equation}\label{eq-bound-of-u-m-p-to-p---1-by-somterm-with-p-3}
p\|w_p\|^{p-1}_{L^{\infty}([t_0,\infty))}\leq \frac{p}{(p-1)t_0}\quad\forall p>1.
\end{equation}
By \eqref{eq-bound-of-u-m-p-to-p---1-by-somterm-with-p-3} for any $R>0$, $T>t_0>0$, we can choose the function $\omega_R$
in Lemma \ref{lemma-L-1-contraction-of-soutions} to be independent of $p\ge 2$. Hence by Lemma \ref{lemma-L-1-contraction-of-soutions} the sequence $\left\{w_{p_i}\right\}_{i=1}^{\infty}$ is equi-continuous in $C([t_0,T];L^1_{loc}(\R))$ for any $T>t_0>0$. Hence by \eqref{L-infty-bd}, the Ascoli theorem and a diagonalization argument the sequence $\left\{w_{p_i}\right\}_{i=1}^{\infty}$ has a subsequence which we may assume without loss of generality to be the sequence itself that converges  in $C([t_0,T];L_{loc}^1\left(\R\right))$ for any $T>t_0>0$ to some non-negative function $w_{\infty}\in C((0,\infty);L_{loc}^1(\R))\cap L^{\infty}(\R\times(0,\infty))$  as $i\to\infty$. Putting $p=p_i$ in \eqref{eq-bound-of-characteristic-solutions-for-each-m} and letting $i\to\infty$, 
\begin{equation}\label{w-infty-bd}
w_{\infty}\leq 1\qquad \mbox{a.e. in $\R\times(0,\infty)$}.
\end{equation}

\begin{lemma}\label{p-infty-limit-integral-eqn-lem}
$w_{\infty}$ satisfies 
\begin{equation}\label{eq-for-u-infty-for-kruzkov-condition-234}
\int_{0}^T\int_{\R}\{|w_{\infty}(x,t)-k|\eta_t+|w_{\infty}(x,t)^m-k^m|\eta_x\}\,dxdt\ge 0\qquad\forall k\in\R,\,\,0\le \eta\in C^{\infty}_0(\R\times(0,\infty)).
\end{equation}
\end{lemma}
\begin{proof}
Let $0\leq\eta\in C^{\infty}_0(\R\times(0,\infty))$. Since $w_p$ is the solution of \eqref{burgers-absorption-eqn}, 
\begin{equation}\label{eq-kruzkov-condition=-for-the-solution-u-mp-p3-4}
\int_{0}^T\int_{\R}\{|w_p(x,t)-k|\eta_t+|w_p(x,t)^m-k^m|\eta_x-\mbox{sign}\,(w_p(x,t)-k)w_p(x,t)^p\eta\}\,dxdt\geq 0\quad\forall k\in\R.
\end{equation} 
We now choose $T>t_0>0$ and $R_1>0$ such that 
\begin{equation*}
\mbox{supp}\,\eta\subset B_{R_1}\times(t_0,T). 
\end{equation*}
By \eqref{eq-bound-of-characteristic-solutions-for-each-m},
\begin{equation}\label{eq-bound-u-m-p-to-p}
w_p(x,t)^p\leq \frac{1}{((p-1)t_0)^{\frac{p}{p-1}}} \qquad \mbox{a.e. $(x,t)\in\R\times [t_0,\infty)$},\,\,\forall p>1.
\end{equation}
Since the right hand side of \eqref{eq-bound-u-m-p-to-p} converges to $0$ as $p\to\infty$, letting $p=p_i$ and $i\to\infty$ in \eqref{eq-kruzkov-condition=-for-the-solution-u-mp-p3-4}, by \eqref{L-infty-bd}, \eqref{eq-bound-u-m-p-to-p} and the Lebesgue Dominated Convergence Theorem, \eqref{eq-for-u-infty-for-kruzkov-condition-234} follows.
\end{proof}

\begin{lemma}\label{lem-inital-convergence-where-u-0-is-=strictly-less-than-1}
Let $0\le u_0\in L^{\infty}(\R)\cap C(\R)$. Suppose there exists $x_0\in\R$ and $\delta>0$ such that
\begin{equation*}
u_0(x)<1, \qquad \forall x\in \overline{B_{2\delta}(x_0)}.
\end{equation*}
Then,
\begin{equation}\label{eq-L-1-norm-of-difference-between-u-m-p-and-u-0-in=distribution-sense}
\lim_{t\to 0}\left\|w_{\infty}(\cdot,t)-u_0(x)\right\|_{L^1(B_{\delta}(x_0))}=0.
\end{equation} 
\end{lemma}

\begin{proof}
We divide the proof into two cases.\\
\textbf{Case 1.} $\left\|u_0\right\|_{L^{\infty}(\R)}\leq 1$.\\
\indent By \eqref{L-infty-bd},
\begin{equation*}
\begin{aligned}
&\left|\int_{\R}w_p(x,t)\eta(x)\,dx-\int_{\R}u_0(x)\eta(x)\,dx\right|\leq \int_{0}^{t}\int_{\R}\left[(w_p)^m|\eta_x|+(w_p)^p|\eta|\right]\,dxdt\le C_{\eta}t, \quad \forall p>1,\,\,\eta\in C^{\infty}_0(\R)\\
&\qquad \Rightarrow \qquad \left|\int_{\R}w_{\infty}(x,t)\eta(x)\,dx-\int_{\R}u_0(x)\eta(x)\,dx\right|\le C_{\eta}t,\qquad \forall\eta\in C^{\infty}_0(\R)\qquad \mbox{as $p=p_i\to\infty$}.
\end{aligned}
\end{equation*}
Then
\begin{equation}\label{eq-weak-convergence-to-u_0-from-w-infty-as-t-to-zero=0}
w_{\infty} \to u_0 \qquad \mbox{weakly in $L^1(\R)$ as $t\to 0$}.
\end{equation}
Let $\left\{t_i\right\}_{i=1}^{\infty}\subset\R^+$ be such that $t_i\to 0$ as $i\to\infty$. Then by \eqref{eq-weak-convergence-to-u_0-from-w-infty-as-t-to-zero=0}, there exists the sequence $\left\{t_i\right\}_{i=1}^{\infty}$has a subsequence which we may assume without loss of generality to be the sequence itself such that
\begin{equation}\label{eq-convergence-of-w-infty-to-u-sub-0034}
w_{\infty}(x,t_i)\to u_0(x) \qquad \mbox{a.e. $x\in\R$} \quad \mbox{as $i\to\infty$}.
\end{equation}
By \eqref{eq-convergence-of-w-infty-to-u-sub-0034} and Lebesgue Dominated Convergence Theorem,
\begin{equation*}
\left\|w_{\infty}(\cdot,t_i)-u_0\right\|_{L^{1}(B_R)}\to 0 \qquad \forall R>0 \qquad \mbox{as $i\to\infty$}.
\end{equation*}
Since the sequence $\left\{t_i\right\}_{i=1}^{\infty}$ is arbitrary, \eqref{eq-L-1-norm-of-difference-between-u-m-p-and-u-0-in=distribution-sense} follows.\\
\textbf{Case 2.} $u_0\in L^{\infty}(\R)$.\\
\indent Let $\theta=\max_{|x-x_0|<\delta}u_0(x)$. Then, $\theta<1$. We now choose a smooth non-negative function $v_0$ on $\R$ such that 
\begin{equation*}
v_0(x)=u_0(x), \quad  \forall x\in B_{2\delta}(x_0) \qquad  \mbox{and} \qquad  v_0(x)\leq \frac{\theta+1}{2}, \quad \forall x\in\R.
\end{equation*} 
Let $N=m\left\|u_0\right\|_{L^{\infty}(\R)}^{m-1}$ and $v_p$ be the solution of \eqref{burgers-absorption-eqn} with initial value $v_0$. By the same argument as before, the sequence $\left\{v_{p_i}\right\}_{i=1}^{\infty}$ has a subsequence which we may assume without loss of generality to be the sequence itself that converges in $C\left([t_0,T);L^1_{loc}(\R)\right)$ for any $T>t_0>0$ to some function $v_{\infty}\in C\left((0,\infty);L^1_{loc}(\R)\right)\cap L^{\infty}(\R\times(0,\infty))$ as $i\to\infty$. Then, By Lemma \ref{lem-L-1-contraction-of-solutions-u-1--and-u-2-2349},
\begin{align}
&\left\|w_p(\cdot,t)-v_p(\cdot,t)\right\|_{L^{1}(B_{\delta}(x_0))}=\left\|w_p(\cdot,0)-v_p(\cdot,0)\right\|_{L^{1}(B_{2\delta}(x_0))}=0 \qquad \forall 0<t<\frac{\delta}{N}, \,\, p>1\notag\\
&\qquad \Rightarrow \qquad \left\|w_{\infty}(\cdot,t)-v_{\infty}(\cdot,t)\right\|_{L^{1}(B_{\delta}(x_0))}=0 \qquad \forall 0<t<\frac{\delta}{N}\quad \mbox{as $p=p_i\to\infty$}\notag\\
&\qquad \Rightarrow \qquad w_{\infty}(x,t)=v_{\infty}(x,t), \qquad \forall 0<t<\frac{\delta}{N},\,\,|x-x_0|\leq \delta. 
\label{eq-absolute-difference-between-w-infty-and-v-infty-as0limit-t-to-zero-and-p-to-infty}
\end{align}
Therefore, by \eqref{eq-absolute-difference-between-w-infty-and-v-infty-as0limit-t-to-zero-and-p-to-infty} and \textbf{Case 1},
\begin{equation*}
\left\|w_{\infty}(\cdot,t)-u_0\right\|_{L^{1}(B_{\delta}(x_0))}=\left\|v_{\infty}(\cdot,t)-u_0\right\|_{L^{1}(B_{\delta}(x_0))}\to 0 \qquad \mbox{as $t\to 0$}
\end{equation*}
and \eqref{eq-L-1-norm-of-difference-between-u-m-p-and-u-0-in=distribution-sense} follows.
\end{proof}

\begin{lemma}\label{p-infty-limit-initial-value-lem}
Let 
\begin{equation*}
w_{\infty}^0(x)=\min\left(u_0(x),1\right)\qquad \forall x\in\R.
\end{equation*} 
Then, 
\begin{equation}\label{eq-initial-function-of-w-infty-by-u-sub-0-and-1-in-two-parts}
\lim_{t\to 0}\int_{B_R}\left|w_{\infty}(x,t)-w_{\infty}^0(x)\right|\,dx=0 \qquad \forall R>0.
\end{equation}
\end{lemma}
\begin{proof}
We divide the proof into 2 cases. \\
\textbf{Case 1.} $u_0\in C(\R)\cap L^{\infty}(\R)$.\\
\indent Since $\left\{x:u_0(x)<1\right\}$ is open, by the Lindelof theorem \cite{R}, $\left\{x:u_0(x)<1\right\}=\bigcup_{j=1}^{\infty}B_{2\delta_j}(x_j)$ for some $x_j\in \left\{x:u_0(x)<1\right\}$ and $\delta_j>0$, $j=1,2,\cdots$. By Lemma \ref{lem-inital-convergence-where-u-0-is-=strictly-less-than-1} for any $j\in\Z^+$ \eqref{eq-L-1-norm-of-difference-between-u-m-p-and-u-0-in=distribution-sense} holds for $\delta=\delta_j$. \\
\indent Let $\3>0$ and $u_{0,\3}(x)=\min(u_0(x),1-\3)$. For any $m>1$, $p>1$, let $u_{m,p,\3}$ be the solutions of \eqref{burgers-absorption-eqn} in $\R\times(0,\infty)$ with initial value $u_{0,\3}$. By the same argument as before $u_{m,p,\3}$ satisfies \eqref{eq-bound-of-u-m-p-to-p---1-by-somterm-with-p-3}. Moreover the sequence $\{u_{m,p_i,\3}\}_{i=1}^{\infty}$ is equi-continuous  in $C([t_0,T];L^{1}_{loc}(\R))$ for any $T>t_0>0$ and has a subsequence which we may assume without loss of  generality to be the sequence itself that converges  in $C([t_0,T];L^{1}_{loc}(\R))$ for any $T>t_0>0$ to some function $w_{\infty,\3}\in C((0,\infty);L^{1}_{loc}(\R))$ as $i\to\infty$ which satisfies
\begin{equation}\label{eq-comparison-betwen-w-infty-epsilon-and-w-infty-zoer}
\begin{aligned}
0\le w_{\infty,\3}(x,t)\le 1\quad\mbox{ in }\R\times[0,\infty)
\end{aligned}
\end{equation}
Since $u_{0,\3}<1$ in $\R$, by the proof of Lemma \ref{lem-inital-convergence-where-u-0-is-=strictly-less-than-1}, 
\begin{equation}\label{eq-initial-condition-or-value-of-w-infty-epsilon-234}
w_{\infty,\3}(x,t)\to u_{0,\3}(x) \qquad \mbox{in $L_{loc}^1(\R)$} \quad \mbox{as $t\to 0$}.
\end{equation}
Since $u_{0,\3}\le u_0$, by the construction of solutions of \eqref{burgers-absorption-eqn} in \cite{K},
\begin{equation}\label{w-infty-w-infty-epsilon-compare}
u_{m,p,\3}\le w_p\quad\mbox{ in }\R\times (0,\infty)\quad\Rightarrow\quad w_{\infty,\3}\le w_{\infty}\quad\mbox{ in }\R\times (0,\infty)\quad\mbox{ as }p=p_i, i\to\infty.
\end{equation}
By \eqref{w-infty-bd}, \eqref{eq-initial-condition-or-value-of-w-infty-epsilon-234} and \eqref{w-infty-w-infty-epsilon-compare},
\begin{align}
&1\ge\limsup_{t\to 0} w_{\infty}(x,t)\ge \liminf_{t\to 0}w_{\infty}(x,t)\ge \lim_{t\to 0}w_{\infty,\3}(x,t)=1-\3 \qquad \mbox{a.e. $x\in\{x:u_0(x)\geq 1\}$}\notag\\
\Rightarrow\quad&\lim_{t\to 0}w_{\infty}(x,t)=1=w_{\infty}^0(x)  \qquad\qquad\qquad \mbox{a.e. $x\in\{x:u_0(x)\geq 1\}$} \qquad \mbox{as $\3\to 0$}.
\label{eq-convetgence-of-w-sub=sinfty-zero-as-to-to-zoer}
\end{align}
Since \eqref{eq-L-1-norm-of-difference-between-u-m-p-and-u-0-in=distribution-sense} holds for $\delta=\delta_j$, $j\in\Z^+$, any sequence $\left\{t_i\right\}_{i=1}^{\infty}$, $t_i\to 0$ as $i\to\infty$, will have a subsequence which we may assume without loss of  generality to be the sequence itself such that
\begin{equation}\label{eq-convergence-of-v-intfyt-tou-sub-0-al-everywhere-1}
w_{\infty}(x,t_i)\to u_0(x) \qquad \mbox{a.e. $x\in\left\{x:u_0(x)<1\right\}$}.
\end{equation}
Hence, by \eqref{eq-convetgence-of-w-sub=sinfty-zero-as-to-to-zoer}, \eqref{eq-convergence-of-v-intfyt-tou-sub-0-al-everywhere-1} and the Lebesgue Dominated Convergence Theorem,
\begin{equation*}
\lim_{i\to\infty}\int_{|x|<R}\left|w_{\infty}(x,t_i)-w_{\infty}^0(x)\right|\,dx=0 \qquad \forall R>0.
\end{equation*}
Since the sequence $\left\{t_i\right\}_{i=1}^{\infty}$ is arbitrary, \eqref{eq-initial-function-of-w-infty-by-u-sub-0-and-1-in-two-parts} follows.\\
\textbf{Case 2.} $u_0\in L^{\infty}(\R)$.\\
\indent We choose a sequence of functions $\{u_{0,j}\}_{j=1}^{\infty}\subset C^{\infty}(\R)$ such that
\begin{equation}\label{eq-cases-aligned-seq-of-u-0-j-in-L-infty}
\begin{cases}
\begin{aligned}
&\|u_{0,j}-u_0\|_{L^{1}(B_R)}\to 0 \qquad \qquad \mbox{as $j\to\infty$},\quad \forall R>0\\
&\qquad u_{0,j}(x)\to u_0(x) \qquad \qquad \mbox{ a.e. }x\in\R\quad\mbox{  as }j\to\infty\\
&\|u_{0,j}\|_{L^{\infty}(\R)}\leq \|u_0\|_{L^{\infty}(\R)}+\frac{1}{j} \quad \forall j\in \Z^+.
\end{aligned}
\end{cases}
\end{equation}
For any $m>1$, $p>1$, let $u_{m,p,j}$ be the solutions of \eqref{burgers-absorption-eqn} with initial value $u_{0,j}$. By the same argument as before for any $j\in\Z^+$ the sequence $\{u_{m,p_i,j}\}_{i=1}^{\infty}$ has a subsequence which we may assume without loss of  generality to be the sequence itself that converges  in $C([t_0,T];L^1_{loc}(\R))$  to some function $w_{\infty,j}\in C((0,\infty);L^{1}_{loc}(\R))$, $0\le w_{\infty,j}\le 1$, for any $T>t_0>0$ as $i\to\infty$. Let
\begin{equation*}
w_{\infty,j}^0(x)=\min\left(u_{0,j}(x),1\right), \quad \forall j\in\Z^+.
\end{equation*} 
By case 1, 
\begin{equation}\label{eq-initial-function-of-w-infty-j-by-u-sub-0-and-1-in-two-parts}
\lim_{t\to 0}\int_{B_R}\left|w_{\infty,j}(x,t)-w_{\infty,j}^0(x)\right|\,dx=0 \qquad \forall R>0, j\in\Z^+.
\end{equation}
By \eqref{eq-cases-aligned-seq-of-u-0-j-in-L-infty} and Lemma \ref{lem-L-1-contraction-of-solutions-u-1--and-u-2-2349} there exists a constant $N>0$ such that 
\begin{equation}\label{eq-applying-L-1-contraction-for-u-m-p-j-epsilon-and-u-m-p-epsilon-298}
\int_{B_{R-Nt}}|u_{m,p,j}(x,t)-u_{m,p}(x,t)|\,dx\le\int_{B_R}|u_{0,j}(x)-u_0(x)|\,dx\quad\forall 0<t<R/N,R>0,j\in\Z^+,p>1.
\end{equation}
Putting $p=p_i$ in \eqref{eq-applying-L-1-contraction-for-u-m-p-j-epsilon-and-u-m-p-epsilon-298} and letting $i\to\infty$, 
\begin{equation}\label{eq-applying-L-1-contraction-for-w-infty-j-and-w-infty-298776}
\int_{B_{R-Nt}}|w_{\infty,j}(x,t)-w_{\infty}(x,t)|\,dx\le\int_{B_R}|u_{0,j}(x)-u_0(x)|\,dx\qquad\forall 0<t<R/N,R>0,j\in\Z^+.
\end{equation}
By \eqref{eq-applying-L-1-contraction-for-w-infty-j-and-w-infty-298776},
\begin{equation}\label{eq-split-of-differnecse-of-w-infty-and-w-infty-0-with-w-infty-0-j}
\begin{aligned}
&\int_{B_{R-Nt}}|w_{\infty}(x,t)-w_{\infty}^0(x)|\,dx\\
&\leq \int_{B_{R-Nt}}|w_{\infty}(x,t)-w_{\infty,j}(x,t)|\,dx+\int_{B_{R-Nt}}|w_{\infty,j}(x,t)-w_{\infty,j}^0(x)|\,dx +\int_{B_{R-Nt}}|w_{\infty,j}^0(x)-w_{\infty}^0(x)|\,dx\\
&\le\int_{B_R}|u_{0,j}(x)-u_0(x)|\,dx+\int_{B_R}|w_{\infty,j}(x,t)-w_{\infty,j}^0(x)|\,dx +\int_{B_R}|w_{\infty,j}^0(x)-w_{\infty}^0(x)|\,dx
\end{aligned}
\end{equation}
for any $0<t<R/N$, $R>0$ and $j\in\Z^+$.
Letting first $t\to 0$  and then $j\to\infty$ in \eqref{eq-split-of-differnecse-of-w-infty-and-w-infty-0-with-w-infty-0-j}, by \eqref{eq-cases-aligned-seq-of-u-0-j-in-L-infty} and \eqref{eq-initial-function-of-w-infty-j-by-u-sub-0-and-1-in-two-parts},  \eqref{eq-initial-function-of-w-infty-by-u-sub-0-and-1-in-two-parts} follows.
\end{proof}

We will now complete the proof of Theorem \ref{main-thm2}.

\begin{proof}[\textbf{Proof of Theorem \ref{main-thm2}}]
 By Lemma \ref{p-infty-limit-integral-eqn-lem} and Lemma \ref{p-infty-limit-initial-value-lem}, $w_{\infty}$ is the unique solution of  \eqref{burgers-eqn}. Since the sequence $\{p_i\}_{i=\infty}^{\infty}$ is arbitrary, $w_p$ converges to $w_{\infty}$ in $C([t_0,T];L_{loc}^1(\R))$ for any $T>t_0>0$ as $p\to\infty$ and Theorem \ref{main-thm2} follows.
\end{proof}

\section{Interchange of limits}
\setcounter{equation}{0}
\setcounter{thm}{0}

This section will be devoted to proving Theorem \ref{main-thm3}.

\begin{proof}[\textbf{Proof of Theorem \ref{main-thm3}}]
Note that (i) follows directly by Theorem \ref{main-thm2} and the result of \cite{X}. Hence we only need to prove (ii). By Theorem \ref{main-thm1}, $u_{\infty,p}$ satisfies \eqref{p-power-ode}
with initial value $u^0_{\infty}$ that satisfies \eqref{p-limit-initial-value} for some function $0\le\psi\in L^1(\R)\cap L^{\infty}(\R)$ which satisfies \eqref{psi=0-eqn} and $0\le u_{\infty,p}\leq 1$ on $\R\times (0,\infty)$. Let $\{p_i\}_{i=1}^{\infty}\subset\Z^+$ be such that $p_i\to\infty$ as $i\to\infty$. Since $0\leq u_{\infty,p}\leq 1$, the sequence $\{u_{\infty,p_i}\}_{i=1}^{\infty}$ has a subsequence which we may assume without loss of generality to be the sequence itself such that $u_{\infty,p_i}$ converges weakly in $L^1(\R\times(0,\infty))$ to some function $v_2$  as $i\to\infty$.

On the other hand since $u_{\infty,p}$ satisfies \eqref{p-power-ode}, 
\begin{align*}
&u_{\infty,p}(x,t)=\frac{u^0_{\infty}(x)}{\left((p-1)tu^0_{\infty}(x)^{p-1}+1\right)^{\frac{1}{p-1}}}\quad\mbox{ a.e. }(x,t)\in\R\times (0,\infty)\quad\forall p>1\\
\Rightarrow\quad&v_2(x,t)=\lim_{p\to\infty}u_{\infty,p}(x,t)=u^0_{\infty}(x)\qquad \qquad \mbox{ a.e. }(x,t)\in\R\times (0,\infty)
\end{align*}
and Theorem \ref{main-thm3} follows.
\end{proof}

\end{document}